\documentclass[12pt,leqno]{amsart}
\usepackage{amsmath}
\usepackage{amssymb}

\def\underset#1#2{{\mathrel{\mathop {{}_{} {#2}}\limits_{{#1}_{}}}}}
\def\upplim_#1{\underset{#1}{\overline\lim}\;}
\def\lowlim_#1{\underset{#1}{\underline\lim}\;}
\newtheorem{claim}[equation]{\indent{\it Claim}\rm }

\newtheorem{lemma}[equation]{Lemma}
\newtheorem{proposition}[equation]{Proposition}

\newtheorem{theorem}[equation]{Theorem}

\newcommand{\C}{{\mathbb{C}}}

\newcommand{\N}{\mathbb{N}}

\renewcommand{\P}{{\mathbb{P}}}

\newcommand{\Z}{\mathbb{Z}}

\numberwithin{equation}{section}

\title[Meromorphic mappings into projective varieties]{Meromorphic mappings into projective varieties intersecting arbitrary families of moving hypersurfaces} 

\author{Si Duc Quang$^{1,2}$ and Nguyen Linh Chi$^{1}$}

\begin{document}

\begin{abstract}
In this paper, we establish a general second main theorem for meromorphic mappings from $\C^m$ into a subvariety $V$ of $\P^n(\C)$ with respect to an arbitrary family of slowly moving hypersurfaces $\mathcal Q=\{Q_1,\ldots,Q_q\}$. In contrast to the usual setting, the mapping is not required to be algebraically nondegenerate over the field $\mathcal K_{\mathcal Q}$. Moreover, the truncation levels of the counting functions are explicitly estimated, and the total defect bound is given by $\Delta_{\mathcal Q,V}(3\dim V-1)$, which is independent of the mapping $f$, where $\Delta_{\mathcal Q,V}$ denotes the distributive constant of $\mathcal Q$ with respect to $V$.
\end{abstract}

\def\thefootnote{\empty}

\maketitle

\section{Introduction}

In 1926, R. Nevanlinna \cite{N} initially established a second main theorem for meromorphic functions and fixed values in the complex plane $\C$. Later, in 1930, H. Cartan \cite{Ca} extended Nevanlinna's result to the case of linearly nondegenerate meromorphic mappings from $\C^m$ into $\P^N(\C)$ with hyperplanes in general position. By introducing the notion of Nochka weights, E. Nochka \cite{Noc83} established a second main theorem for such mappings with families of hyperplanes in subgeneral position.

In recent years, the second main theorem for hypersurfaces (fixed or moving) has been intensively studied by many authors. First, in 1992, A. E. Eremenko and M. L. Sodin \cite{ES} proved a second main theorem for arbitrary holomorphic curves from $\C$ into $\P^N(\C)$ intersecting $q$ hypersurfaces $\{Q_i\}_{i=1}^q$ in general position. In 2004, M. Ru \cite{Ru04} considered the case of algebraically nondegenerate meromorphic mappings from $\C^m$ into $\P^N(\C)$ with families of hypersurfaces in general position by using the filtration of the vector space of homogeneous polynomials introduced by P. Corvaja and U. Zannier \cite{CZ}. Then, in \cite{Ru09}, he generalized his result to the case of algebraically nondegenerate meromorphic mappings into a projective subvariety of $\P^N(\C)$ by using the method in Diophantine approximation proposed by J. Evertse and R. Ferretti \cite{EF1}. His results were further generalized by S. D. Quang to the case of families of hypersurfaces in subgeneral position in \cite{Q19}, where he proposed the replacing hypersurfaces method in order to avoid using Nochka weights.

On the other hand, the results of M. Ru were also generalized to the case of moving hypersurfaces by G. Dethloff and T. V. Tan \cite{DT1,DT2}. In order to establish the second main theorem for moving hypersurfaces in general position, G. Dethloff and T. V. Tan constructed a new filtration instead of the filtration of Corvaja and Zannier. Also, in 2018, S. D. Quang \cite{Q18} established the second main theorem for moving hypersurfaces in subgeneral position. Recently, Q. Yan and G. Yu \cite{YY} used the methods of \cite{DT2,Q18,Q19} to obtain a second main theorem for nonconstant meromorphic mappings, but their theorem does not have the usual normal form of Cartan's theorem. More recently, S. D. Quang \cite{Q22c} generalized all the above-mentioned results by considering arbitrary families of moving hypersurfaces. In order to state this result, we need the following preparation.

We recall the following notions. Denote by $\mathcal M$ the field of all meromorphic functions on $\C^m$. A moving hypersurface of degree $d$ in $\P^n(\C)$ is a homogeneous polynomial $Q\in\mathcal M[x_0,\ldots,x_n]$ of the form
$$
Q(z)({\bf x})=\sum_{I\in\mathcal T_{d}}a_{I}(z){\bf x}^I,
$$
where $\mathcal T_d$ is the set of all $(n+1)$-tuples $(i_0,\ldots,i_n)$ satisfying $i_0+\cdots+i_n=d$ and $i_j\ge 0$ for all $j$, $a_{I}\in\mathcal M\ (I\in\mathcal T_{d})$ are not all identically zero, ${\bf x}=(x_0,\ldots,x_n)$, and ${\bf x}^I=x_0^{i_0}\cdots x_n^{i_n}$ for all $I\in\mathcal T_d$.
For each $z$ which is neither a pole of any $a_{I}$ nor a common zero of all $a_{I}\ (I\in\mathcal T_{d})$, we denote by $Q(z)^*$ the support of $Q(z)$, namely,
$$
Q(z)^*=\{(x_0:\cdots:x_n)\in\P^n(\C)\mid \sum_{I\in\mathcal T_{d}}a_{I}(z){\bf x}^I=0\}.
$$
The moving hypersurface $Q$ is said to be slow with respect to a meromorphic mapping $f$ from $\C^m$ into $\P^n(\C)$ if
$$
\| \ T(r,a_I)=o(T_f(r))\quad \forall I\in\mathcal T_d,
$$
where $T(r,a_I)$ and $T_f(r)$ are the characteristic functions of $a_I$ and $f$ respectively (see Section 2 for the definitions).
Here, the notation ``$\| P$'' means that the assertion $P$ holds for all $r\in(0,+\infty)$ outside a set of finite Borel measure.

Let $\mathcal Q=\{Q_i\}_{i=1}^q$ be a family of moving hypersurfaces in $\P^n(\C)$ given by
$$
Q_i(z)({\bf x})=\sum_{I\in\mathcal T_{d_i}}a_{iI}(z){\bf x}^I,
$$
where $d_i=\deg Q_i\ (1\le i\le q)$. Denote by ${\mathcal K}_{\mathcal Q}$ the smallest subfield of $\mathcal M$ containing $\C$ and all ratios $\frac{a_{iI}}{a_{iJ}}$ with $a_{iJ}\not\equiv 0$.

Let $\tilde f=(f_0,\ldots,f_n)$ be a reduced representation of $f$. Denote by $I(f(\C^m))$ the ideal of all polynomials $Q$ in ${\mathcal K}_{\mathcal Q}[x_0,\ldots,x_n]$ (including the zero polynomial) such that $Q(\tilde f)\equiv 0$. Choose a minimal generating set of $I(f(\C^m))$, denoted by $\{R_1,\ldots,R_p\}$. The algebraic dimension of $f$ over $\mathcal K_{\mathcal Q}$ is defined by
$$
n_f:=\dim V_z,\ \text{where }\ V_z=\bigcap_{i=1}^pR_i(z)^*,
$$
for generic points $z\in\C^m$. Here, whenever we say that an assertion holds for generic points, we mean that it holds outside a proper analytic subset of $\C^m$.

Let $H_{V_z}(u)$ be the Hilbert function of $V_z$ (see Section 2 for the definition). For a fixed positive integer $u$, $H_{V_z}(u)$ is constant for generic points $z$. By the usual theory of Hilbert polynomials,
\begin{align*}
H_{V_z}(u)=\deg V_z\cdot\dfrac{u^{\dim V_z}}{(\dim V_z)!}+O(u^{\dim V_z-1}).
\end{align*}
This implies that $\deg V_z$ is also constant for generic points $z$. We define the algebraic degree of $f$ over $\mathcal K_{\mathcal Q}$ by
$$
\delta_f:=\deg V_z,
$$
for generic points $z$.

Let $V$ be a projective subvariety of $\P^n(\C)$. For an analytic subset $S$ of $\P^n(\C)$, we define
$$
\mathrm{codim}_VS:=
\begin{cases}
\dim V-\dim (V\cap S),&\text{if }V\cap S\ne\varnothing,\\
\dim V+1,&\text{if }V\cap S=\varnothing.
\end{cases}
$$

We now recall the following definition from \cite{Q22c}.

\vskip0.2cm
\noindent
{\bf Definition A.} (cf. \cite[Definition 3.3]{Q22c})
{\it With the above notation, we define the distributive constant of the family $\mathcal Q=\{Q_1,\ldots,Q_q\}$ with respect to $f$ by
$$\Delta_{\mathcal Q,f}:=\underset{\varnothing\ne\Gamma\subset\{1,\ldots,q\}}{\max}\dfrac{\sharp\Gamma}{\mathrm{codim}_{V_z}\left(\bigcap_{j\in\Gamma}Q_j(z)^*\right)}$$
for generic points $z\in\C^m$.}

In \cite{Q22c}, the first author proved the following theorem.

\vskip0.2cm
\noindent
{\bf Theorem B.} (see \cite[Theorem 1.1]{Q22c})
{\it Let $f$ be a meromorphic mapping of $\C^m$ into $\P^n(\C)$. Let $\mathcal Q=\{Q_i\}_{i=1}^q$ be a family of slowly moving hypersurfaces in $\P^n(\C)$ with respect to $f$, and let $n_f$ denote the algebraic dimension of $f$ over $\mathcal K_{\mathcal Q}$. Then, for any $\epsilon>0$,
$$\|\ (q-\Delta_{\mathcal Q,f}(n_f+1)-\epsilon)T_f(r)\le\sum_{i=1}^{q}\dfrac{1}{\deg Q_i}N_{(Q_i,f)}(r).$$}
Here, $N_{(Q_i,f)}(r)$ denote the counting function of the divisor $f^*Q_i$ (see Section 2 for the definition).

Hence, the total defect bound obtained in this theorem is $\Delta_{\mathcal Q,f}(n_f+1)$. If $f$ is a meromorphic mapping from $\C^m$ into a projective subvariety $V\subset\P^n(\C)$ which is algebraically nondegenerate over $\mathcal K_{\mathcal Q}$, and if $\mathcal Q$ is in general position with respect to $V$, then $n_f=\dim V$ and $\Delta_{\mathcal Q,f}=1$. Therefore, the bound becomes $\dim V+1$, and we recover the second main theorem of Dethloff and Tan in \cite{DT2}. 

However, without imposing any algebraic nondegeneracy condition on $f$, the bound $\Delta_{\mathcal Q,f}(n_f+1)$ depends on $f$ and is generally difficult to estimate. In this paper, we establish a second main theorem whose total defect bound does not depend on the mapping, and whose counting functions are truncated at an explicit level. To state our result, we recall the following definition.

\vskip0.2cm
\noindent
{\bf Definition C.} (see \cite[Definition 3.4]{Q22c})
{\it Let $V$ be a subvariety of $\P^n(\C)$ of dimension $k$. Assume that $V$ is not generically contained in any $Q_j(z)^*\ (1\le j\le q)$. We define the distributive constant of $\{Q_1,\ldots,Q_q\}$ with respect to $V$ by
$$
\Delta_{\mathcal Q,V}:=\underset{\varnothing\ne\Gamma\subset\{1,\ldots,q\}}{\max}\dfrac{\sharp\Gamma}{\mathrm{codim}_V\left(\bigcap_{j\in\Gamma}Q_j(z)^*\right)}
$$
for generic points $z\in\C^m$.}

The family $\mathcal Q$ is said to satisfy the B\'{e}zout condition with respect to $f$ if for and subsets $\Gamma_1,\Gamma_2$ of $\{1,\ldots,q\}$,
$$ \mathrm{codim}_{V_z}\left(\bigcap_{j\in\Gamma_1\cup\Gamma_2}Q_j(z)^*\right)\le \mathrm{codim}_{V_z}\left(\bigcap_{j\in\Gamma_1}Q_j(z)^*\right)+\mathrm{codim}_{V_z}\left(\bigcap_{j\in\Gamma_2}Q_j(z)^*\right)$$
for generic points $z$. Note that, this condition is automatically satisfied if $V_z=\P^n(\C).$

In this paper, we prove the following theorem.

\begin{theorem}\label{1.1}
Let $f$ be a nonconstant meromorphic map of $\mathbf{C}^m$ into a projective subvariety $V\subset\P^n(\mathbf{C})$ of dimension $k$. Let $\mathcal Q=\{Q_i\}_{i=1}^q$ be a family of slowly (with respect to $f$) moving hypersurfaces in $\P^n(\C)$ satisfying the B\'{e}zout condition with respect to $f$ and $d=lcd(\deg Q_1,\ldots,\deg Q_q)$. Let $\Delta_{\mathcal Q,V}$ be the distributive constant of $\mathcal Q$ with respect to $V$ and $\delta_f$ the algebraic degree of $f$ over $\mathcal K_{\mathcal Q}$. Then for any $\epsilon >0$, 
$$  \|\ (q-\Delta_{\mathcal Q,V}(3k-1)-\epsilon)T_f(r)\le \sum_{i=1}^{q}\dfrac{1}{\deg Q_i}N^{[L-1]}_{(Q_i,f)}(r),$$
where $$L=d^k\delta_f(u+1)^k\left\lfloor\biggl(1+\frac{\epsilon}{4}\biggl)\cdot\mathrm{exp}\left(d^k\delta_f(u+1)^{k+q}\left(\frac{1}{2}+\frac{4k\Delta_{Q,V}}{\epsilon}\right)\right)\right\rfloor$$
with $u=\lceil 2k(2k+1)\Delta_{Q,V}d^k\delta_f(2k\Delta_{Q,V}+\epsilon)\epsilon^{-1}\rceil$. Moreover, if all $Q_i\ (1\le i\le q)$ are fixed hypersurfaces then we may take 
$$L=\left\lfloor d^{k^2+k}(\delta_f)^{k+1}e^{k+\frac{3}{2}}(4k)^k\Delta_{\mathcal Q,V}^k(2k\Delta_{\mathcal Q,V}\epsilon^{-1}+1)^k\right\rfloor.$$
\end{theorem} 

Here, $\left\lceil x\right\rceil$ is the smallest integer greater than or equal to the real number $x$, and $\lfloor x\rfloor$ is the the greatest integer not exceeding $x$.

If the family $\mathcal Q$ is in $N$-subgeneral position with respect to $V$ then we will prove a more optimal result as follows.
\begin{theorem}\label{1.2}
Let $f$ be a nonconstant meromorphic map of $\mathbf{C}^m$ into a projective subvariety $V\subset\P^n(\mathbf{C})$ of dimension $k$. Let $\mathcal Q=\{Q_i\}_{i=1}^q$ be a family of slowly (with respect to $f$) moving hypersurfaces in $\P^n(\C)$ in $N$-subgeneral position with respect to $V$ satisfying the B\'{e}zout condition with respect to $f$ and $d=lcd(\deg Q_1,\ldots,\deg Q_q)$. Let $\delta_f$ the algebraic degree of $f$ over $\mathcal K_{\mathcal Q}$. Then for any $\epsilon >0$, 
$$  \|\ (q-3N+1-\epsilon)T_f(r)\le \sum_{i=1}^{q}\dfrac{1}{\deg Q_i}N^{[L-1]}_{(Q_i,f)}(r),$$
where $$L=d^k\delta_f(u+1)^k\left\lfloor\biggl(1+\frac{\epsilon}{4}\biggl)\cdot\mathrm{exp}\left(d^k\delta_f(u+1)^{k+q}\left(\frac{1}{2}+\frac{4N}{\epsilon}\right)\right)\right\rfloor$$
with $u=\lceil 2N(2k+1)d^k\delta_f(2N+\epsilon)\epsilon^{-1}\rceil$. Moreover, if all $Q_i\ (1\le i\le q)$ are fixed hypersurfaces then we may take 
$$L=\left\lfloor d^{k^2+k}(\delta_f)^{k+1}e^{k+\frac{3}{2}}(4N)^k(2N\epsilon^{-1}+1)^k\right\rfloor.$$
\end{theorem} 
Note that, in the case where $f:\C^m\rightarrow\P^n(\C)$ is algebraically nondegenerate over $\mathcal K_f$ (the field of all meromorphic functions which are small with respect to $f$), and $\mathcal Q$ is in $N$-subgeneral position, Q. Cai and C.J. Yang \cite{CY} obtained the total defect bound $\frac{3}{2}(2N-n+1).$ Furthermore, if $\mathcal Q$ is in general position and satisfies the B\'ezout condition with respect to $f$, and if $n\le 8$, then L. Shi, Q. Yan, and G. Yu \cite{SYY} obtained the total defect bound $2n.$ However, all of these results were established without truncation levels for the counting functions.

\section{Notation and Auxialiary results}

\noindent
\textbf{(a) Some notation from Nevanlina theory.}

For $z = (z_1,\dots,z_m) \in \mathbb C^m$, we set $ \|z\|:= \big(|z_1|^2 + \dots + |z_m|^2\big)^{1/2}$ and define
\begin{align*}
v_{m-1}(z) &:= \big(dd^c \|z\|^2\big)^{m-1},\\
\sigma_m(z)&:= d^c\log\|z\|^2 \land \big(dd^c\log\|z\|^2\big)^{m-1} \text{on} \quad \mathbb C^m \setminus \{0\}.
\end{align*}
 For a divisor $\nu$ on  a $\mathbb C^m$ and a positive integer $M$ or $M= +\infty$, as usual we denote by $N^{[M]}(r,\nu)$ the counting function of $\nu$ with multiplicities truncated to level $M$.

For a meromorphic function $\varphi$ on $\C^m$, denote by $\nu_\varphi$ its divisor of zeros and set
$$N_{\varphi}(r)=N(r,\nu_{\varphi}), \ N_{\varphi}^{[M]}(r)=N^{[M]}(r,\nu_{\varphi})\ (r_0<r<R_0).$$

For brevity, we will omit the character $^{[M]}$ if $M=+\infty$.

Let $f : \mathbf{C}^m \longrightarrow \P^n(\mathbf{C})$ be a meromorphic mapping with a reduced representation
$\tilde f= (f_0,  \ldots , f_n)$. Set $\|\tilde f\| = \big(|f_0|^2 + \dots + |f_n|^2\big)^{1/2}$.
The characteristic function of $f$ is defined by 
\begin{align*}
T_f(r)= \int\limits_{\|z\|=r} \log\|\tilde f\| \sigma_m -\int\limits_{\|z\|=1}\log\|\tilde f\|\sigma_m.
\end{align*}

Let $\varphi$ be a nonzero meromorphic function on $\mathbf{C}^m$, which are occasionally regarded as a meromorphic map into $\P^1(\mathbf{C})$. The proximity function of $\varphi$ is defined by
$$m(r,\varphi):=\int_{\|z\|=r}\log \max\ (|\varphi|,1)\sigma_m.$$
The Nevanlinna's characteristic function of $\varphi$ is defined as follows
$$ T(r,\varphi):=N_{\frac{1}{\varphi}}(r)+m(r,\varphi). $$
Then 
$$T_\varphi (r)=T(r,\varphi)+O(1).$$
The function $\varphi$ is said to be small (with respect to $f$) if $\|\ T_\varphi (r)=o(T_f(r))$.

Let $Q$ be a moving hypersurface in $\P^n(\mathbf{C})$ of degree $d\ge 1$ given by
$$ Q(z)({\bf x})=\sum_{I\in\mathcal T_d}a_I{\bf x}^I, $$
where $a_I\in\mathcal M$ for all $I\in\mathcal T_d$. We will write $N^{[M]}_{(Q,f)}(r)$ for $N^{[M]}_{Q(\tilde f)}(r)$.
The proximity function of $f$ with respect to $Q$, denoted by $m_f (r,Q)$, is defined by
$$m_f (r,Q)=\int_{\|z\|=r}\log\dfrac{\|\tilde f\| ^d}{|Q(\tilde f)|}\sigma_m-\int_{\|z\|=1}\log\dfrac{\|\tilde f\| ^d}{|Q(\tilde f)|}\sigma_m.$$
This definition is independent of the choice of the reduced representation of $f$. 

If $Q$ is a slowly moving hypersurface with respect to $f$, then  the first main theorem in Nevanlinna theory for meromorphic mappings and moving hypersurfaces states
$$dT_f (r)=m_f (r,Q) + N_{Q(\tilde f)}(r)+o(T_f(r)).$$

\begin{lemma}[{Lemma on logarithmic derivative, see \cite{NO}}]\label{2.1}
Let $f$ be a nonzero meromorphic function on $\mathbf{C}^m$, which is consider as a meromorphic mapping into $\P^1(\C)$. Then 
$$\biggl|\biggl|\quad m\biggl(r,\dfrac{\mathcal{D}^\alpha (f)}{f}\biggl)=O(\log^+T_f(r))\ (\alpha\in \Z^m_+).$$
\end{lemma}

Repeating the argument in (Prop. 4.5 \cite{Fu}), we have the following.

\begin{proposition}[{see \cite[Prop. 4.5]{Fu}}]\label{2.2}
Let $\Phi_1,...,\Phi_k$ be meromorphic functions on $\mathbf{C}^m$ such that $\{\Phi_1,...,\Phi_k\}$ 
are  linearly independent over $\mathbf{C}.$ Then  there exists an admissible set, which is uniquely chosen in an explicitly way,  
$$\{\alpha_i=(\alpha_{i1},...,\alpha_{im})\}_{i=1}^k \subset \Z^m_+$$
with $|\alpha_i|=\sum_{j=1}^{m}|\alpha_{ij}|\le i-1 \ (1\le i \le k)$ such that the following are satisfied:

(i)\  $\{{\mathcal D}^{\alpha_i}\Phi_1,...,{\mathcal D}^{\alpha_i}\Phi_k\}_{i=1}^{k}$ is linearly independent over $\mathcal M,$\ i.e., \ $\det{({\mathcal D}^{\alpha_i}\Phi_j)}\not\equiv 0,$ 

(ii) $\det \bigl({\mathcal D}^{\alpha_i}(h\Phi_j)\bigl)=h^{k}\cdot \det \bigl({\mathcal D}^{\alpha_i}\Phi_j\bigl)$ for
any nonzero meromorphic function $h$ on $\mathbf{C}^m.$
\end{proposition}
The determinant $\det \bigl({\mathcal D}^{\alpha_i}\Phi_j\bigl)$ is usually called the general Wronskian of $\{\Phi_1,...,\Phi_k\}$.

\vskip0.2cm
\noindent
\textbf{(b) Some auxiliary results.}  Let $X\subset\P^n(\C)$ be a projective variety of dimension $k$ and of degree $\delta$. For $\textbf{a} = (a_0,\ldots,a_n)\in\mathbb Z^{n+1}_{\ge 0}$ we write ${\bf x}^{\bf a}$ for the monomial $x^{a_0}_0\cdots x^{a_n}_n$. Denote by $I_X$ the the prime ideal in $\C[x_0,\ldots,x_n]$ defining $X$ and by $\C[x_0,\ldots,x_n]_u$ the vector space of homogeneous polynomials in $\C[x_0,\ldots,x_n]$ of degree $u$ (including $0$). For $u= 1, 2,\ldots,$ set $(I_X)_u:=\C[x_0,\ldots,x_n]_u\cap I_X$ and define the Hilbert function $H_X$ of $X$ by
\begin{align*}
H_X(u):=\dim\C[x_0,\ldots,x_n]_u/(I_X)_u.
\end{align*}
For a tuple ${\bf c}=(c_0,\ldots,c_n)$ in $\mathbb R^{n+1}_{\ge 0}$, denote by $e_X({\bf c})$ the Chow weight of $X$ with respect to ${\bf c}$. The $u$-th Hilbert weight $S_X(u,{\bf c})$ of $X$ with respect to ${\bf c}$ is defined by
\begin{align*}
S_X(u,{\bf c}):=\max\sum_{i=1}^{H_X(u)}{\bf a}_i\cdot{\bf c},
\end{align*}
where the maximum is taken over all sets of monomials ${\bf x}^{{\bf a}_1},\ldots,{\bf x}^{{\bf a}_{H_X(u)}}$ whose residue classes modulo $I_X$ form a basis of $\C[x_0,\ldots,x_n]_u/(I_X)_u.$

The following theorem is due to J. Evertse and R. Ferretti \cite{EF1}.
\begin{theorem}[{see \cite[Theorem 4.1]{EF1}}]\label{2.1}
Let $X\subset\P^n(\C)$ be an algebraic variety of dimension $k$ and degree $\delta$. Let $u>\delta$ be an integer and let ${\bf c}=(c_0,\ldots,c_n)\in\mathbb R^{n+1}_{\geqslant 0}$.
Then
$$ \dfrac{1}{uH_X(u)}S_X(u,{\bf c})\ge\dfrac{1}{(k+1)\delta}e_X({\bf c})-\dfrac{(2k+1)\delta}{u}\cdot\left (\max_{i=0,\ldots,n}c_i\right).$$
\end{theorem}

Let $\mathcal Q=\{Q_1,\ldots,Q_q\}$ be $q$ moving hypersurfaces in $\P^n(\C)$. Denote by $\mathcal C_{\mathcal Q}$ the set of all non-negative functions $h : \mathbf{C}^m\setminus A\longrightarrow [0,+\infty]$, which are of the form
$$ h=\dfrac{|g_1|+\cdots +|g_l|}{|g_{l+1}|+\cdots +|g_{l+k}|}, $$
where $k,l\in\N,\ g_1,...., g_{l+k}\in\mathcal K_{\mathcal Q}\setminus\{0\}$ and $A\subset\mathbf{C}^m$, which may depend on
$g_1,....,g_{l+k}$, is an analytic subset of codimension at least two. Then, for every moving hypersurface $Q$ in $\mathcal K_{\mathcal Q}[x_0,\ldots,x_n]$ of degree $d$, we have
$$ Q(z)({\bf x})\le c(z)\|{\bf x}\|^d $$
for some $c\in\mathcal C_{\mathcal Q}$.

Moreover, if $f$ is a meromorphic mapping from $\C^m$ into $\P^n(\C)$ and $V_z=\bigcap_{i=1}^pR_i(z)^*$ for generic points $z$, where $\{R_1,\ldots,R_q\}$ is the minimal generating set of $I(f(\C^m))$, then we have the following lemma. 

\begin{lemma}[{see \cite[Lemma 3.2]{Q22c}}]\label{2.4}
With the above notation, let $1\le j_1\le\cdots\le j_k\le q$. Suppose that there exists $z_0\in\C^m$ such that $V_{z_0}\cap\bigcap_{s=1}^kQ_{j_s}(z_0)^*=\emptyset.$ Then we have $V_z\cap\bigcap_{s=1}^kQ_{j_s}(z)^*=\emptyset$ for generic points $z$, and there exists a function $c\in\mathcal C_{\mathcal Q}$ such that
$$ \|\tilde f(z)\|\le c(z)\max_{1\le s\le k}\{Q_{j_s}(\tilde f)(z)\}.$$
\end{lemma}

The following lemma is due to the first author \cite{Q22b}.
\begin{lemma}[{see \cite[Lemma 3.2]{Q22b}}]\label{2.2}
Let $Y$ be a projective subvariety of $\P^R(\C)$ of dimension $k\ge 1$ and degree $\delta_Y$. Let $\ell\ (\ell\ge k+1)$ be an integer and let ${\bf c}=(c_0,\ldots,c_R)$ be a tuple of non-negative reals. Let $\mathcal H=\{H_0,\ldots,H_R\}$ be a set of hyperplanes in $\P^R(\C)$ defined by $H_{i}=\{y_{i}=0\}\ (0\le i\le R)$. Let $\{i_1,\ldots, i_\ell\}$ be a subset of $\{0,\ldots,R\}$ such that:
\begin{itemize}
\item[(1)] $c_{i_\ell}=\min\{c_{i_1},\ldots,c_{i_\ell}\}$,
\item[(2)] $Y\cap\bigcap_{j=1}^{\ell-1}H_{i_j}\ne \emptyset$, 
\item[(3)] and $Y\not\subset H_{i_j}$ for all $j=1,\ldots,\ell$.
\end{itemize}
Let $\Delta_{\mathcal H,Y}$ be the distributive constant of the family $\mathcal H=\{H_{i_j}\}_{j=1}^\ell$ with respect to $Y$. Then
$$e_Y({\bf c})\ge \frac{\delta_Y}{\Delta_{\mathcal H,Y}}(c_{i_1}+\cdots+c_{i_\ell}).$$
\end{lemma}

\section{Meromorphic mappings into $\P^n(\C)$ with arbitrary families of moving hypersurfaces}

In this section, we establish a second main theorem for meromorphic mappings into $\P^n(\C)$ with respect to arbitrary families of moving hypersurfaces, where the truncation levels are explicitly estimated and the mappings are allowed to be algebraically degenerate. In order to prove this result, we first derive a general form of the second main theorem for moving hyperplanes. We begin by recalling the following notions.

For a subset $\Psi\subset \mathcal{M}$, denote by $\mathcal L(\Psi)$ the $\C$-vector space generated by $\Psi$ over $\C$. Assume that $q:=\sharp\Psi<\infty$ and that $1\in\Psi$. For each positive integer $p$, define
$$
\Psi(p)=\{\varphi_1\varphi_2\cdots\varphi_p\mid \varphi_j\in\Psi,\ j=1,\ldots,p\}.
$$
Then
$$
1\in\Psi(p),\quad \Psi(p)\subset \Psi(p+1),\quad \sharp\Psi(p)=\binom{p+q-1}{p}=\binom{p+q-1}{q-1}.
$$

Let $0<\epsilon<1$ be arbitrarily fixed. Then there exists a smallest positive integer, which will always be denoted by $p$ throughout this section, such that
$$
\dfrac{\dim \mathcal L(\Psi(p+1))}{\dim \mathcal L(\Psi(p))}\le 1+\epsilon.
$$

\noindent
{\bf Remark.} In \cite{QT10}, the following estimate was obtained:
\begin{align}\label{4.1}
\dim\mathcal L(\Psi(p+1))\le\binom{p+q-1}{q-1}\le\left\lfloor(1+\epsilon)^{\left\lfloor\frac{q}{\log^2(1+\epsilon)}\right\rfloor+1}\right\rfloor.
\end{align}

\begin{theorem}\label{main}  
Let $f$ be a nonconstant meromorphic map of $\mathbf{C}^m$ into $\P^n(\mathbf{C})$. Let $\mathcal Q=\{Q_i\}_{i=1}^q$ be a family of slowly (with respect to $f$) moving hypersurfaces in $\P^n(\C)$. Let $n_f$ and $\delta_f$ be the algebraic dimension and algebraic degree of $f$ over $\mathcal K_{\mathcal Q}$ respectively.  Let $\Delta_{\mathcal Q,f}$ be the distributive constant of $\mathcal Q$ with respect to $f$ and $d=lcm(\deg Q_1,\ldots,\deg Q_q)$. Then for any $\epsilon >0$, 
$$  \|\ (q-\Delta_f(n_f+1)-\epsilon)T_f(r)\le \sum_{i=1}^{q}\dfrac{1}{d_i}N^{[L-1]}_{(Q_i,f)}(r),$$
where 
$$L=d^{n_f}\delta_f(u+1)^{n_f}\left\lfloor\biggl(1+\frac{\epsilon}{2({n_f}+1)\Delta_{\mathcal Q,f}}\biggl)^{\left\lfloor\frac{d^{n_f}\delta_f(u+1)^{n_f+q}}{\log^2(1+\frac{\epsilon}{2(n_f+1)\Delta_{\mathcal Q,f}})}\right\rfloor+1}\right\rfloor$$
with $u=\lceil 2\Delta_{\mathcal Q,f}(2n_f+1)(n_f+1)d^n\delta_f(\Delta_{\mathcal Q,f}(n_f+1)+\epsilon)\epsilon^{-1}\rceil$. Moreover, if all $Q_i\ (1\le i\le q)$ are fixed hypersurfaces then we may take 
$$L=\left\lfloor d^{n^2_f+n_f}(\delta_f)^{n_f+1}e^{n_f}(2n_f+5)^{n_f}(\Delta^2_{\mathcal Q,f}(n_f+1)\epsilon^{-1}+\Delta_{\mathcal Q,f})^{n_f}\right\rfloor.$$
\end{theorem}

\begin{proof}
Let $\tilde f=(f_0,\ldots,f_n)$ be a reduced representation of $f$, and let $I(f(\C^m))$ denote the ideal of the ring $\mathcal K_{\mathcal Q}[x_0,\ldots,x_n]$ consisting of all polynomials $P\in \mathcal K_{\mathcal Q}[x_0,\ldots,x_n]$ (including the zero polynomial) such that
$$P(\tilde f)(z)\equiv 0.$$
Let $\{R_1,\ldots,R_\lambda\}$ be a minimal generating set of the ideal $I(f(\C^m))$.

For every point $z$ which is not a pole of any $R_i\ (1\le i\le \lambda)$, define $V_z=\bigcap_{i=1}^\lambda R_i(z)^*.$
Then, for generic points $z$, one has
$$\dim V_z=n_f,\qquad \deg V_z=\delta_f.$$

Replacing each $Q_j$ by $Q_j^{\frac d{\deg Q_j}}$ if necessary, we may assume that
$$\deg Q_j=d,\qquad 1\le j\le q.$$
Let $\sigma_1,\ldots,\sigma_{n_0}$ be all bijections from $\{1,\ldots,q\}$ onto itself. Then, for each $\sigma_i\ (1\le i\le n_0)$, there exists a smallest index $\ell_i\le q$ such that
$$\bigcap_{j=1}^{\ell_i}Q_{\sigma_i(j)}(z)^*\cap V_z=\emptyset$$
for generic points $z$.
Denote by $\mathcal S$ the set of all points $z\in\C^m$ satisfying at least one of the following conditions:
\begin{itemize}
\item $z$ is a pole of some coefficient of $R_i$ or $Q_j$;
\item $\bigcap_{j=1}^{\ell_i}Q_{\sigma_i(j)}(z)^*\cap V_z\ne\emptyset$ for some $i$;
\item $Q_j(\tilde f)(z)=0$ for some $j$.
\end{itemize}
Then $\mathcal S$ is a proper analytic subset of $\C^m$.

For each $i\ (1\le i\le \sigma_i)$, let $S(i)$ be the set of all points $z\in\C^m\setminus\mathcal S$ such that
$$
\frac{|Q_{\sigma_i(1)}(\tilde f)(z)|}{\|Q_{\sigma_i(1)}(z)\|}
\le
\frac{|Q_{\sigma_i(2)}(\tilde f)(z)|}{\|Q_{\sigma_i(2)}(z)\|}
\le\cdots\le
\frac{|Q_{\sigma_i(q)}(\tilde f)(z)|}{\|Q_{\sigma_i(q)}(z)\|}.
$$
Then, for every $z\in S(i)$,
$$
\prod_{j=1}^q
\frac{\|\tilde f(z)\|^d\|Q_j(z)\|}{|Q_j(\tilde f)(z)|}
\le
C(z)\prod_{j=1}^{\ell_i}
\frac{\|\tilde f(z)\|^d\|Q_{\sigma_i(j)}(z)\|}{|Q_{\sigma_i(j)}(\tilde f)(z)|},
$$
where $C\in\mathcal C_{\mathcal Q}$.

For each $z\notin\mathcal S$, consider the mapping $\Phi_z:V_z\longrightarrow \P^{q-1}(\C)$ defined by
$$
\Phi_z(x)=\bigl(Q_1(z)(\tilde x):\cdots:Q_q(z)(\tilde x)\bigr)
$$
for every $x=(x_0:\cdots:x_n)\in V_z$ and $\tilde x=(x_0,\ldots,x_n).$
We also set
$$
\tilde\Phi_z(x)=\bigl(Q_1(z)(x),\ldots,Q_q(z)(x)\bigr).
$$
Let $Y_z=\Phi_z(V_z).$ Since $\bigcap_{j=1}^q(Q_j(z)^*\cap V_z)=\emptyset,$
the map $\Phi_z$ is a finite morphism on $V_z$, and $Y_z$ is a projective subvariety of $\P^{q-1}(\C)$ satisfying
$$\dim Y_z=n_f\quad\text{ and }\quad\delta:=\deg Y_z\le d^{n_f}\deg V_z=d^{n_f}\delta_f$$
for generic points $z$.

For every ${\bf a}=(a_1,\ldots,a_q)\in\mathbb Z_{\ge0}^q$ and ${\bf y}=(y_1,\ldots,y_q),$
we write
$${\bf y}^{\bf a}=y_1^{a_1}\cdots y_q^{a_q}.$$

Let $u$ be a positive integer and set $\xi_u=\binom{q+u}{u}.$
Define the $\C$-vector space
$$Y_{z,u}:=\C[y_1,\ldots,y_q]_u/(I_{Y_z})_u.$$
Denote by $(I_Y)_u$ the subspace of the $\mathcal K_{\mathcal Q}$-vector space $\mathcal K_{\mathcal Q}[y_1,\ldots,y_q]_u$
consisting of all homogeneous polynomials $P\in\mathcal K_{\mathcal Q}[y_1,\ldots,y_q]_u$ (including the zero polynomial) such that
$$P(z)(\Phi_z(\tilde f(z)))\equiv 0.$$

Let $(\tilde R_1,\ldots,\tilde R_p)$ be a $\mathcal K_{\mathcal Q}$-basis of $(I_Y)_u$. Enlarging $\mathcal S$ if necessary, we may assume that all zeros and poles of every nonzero coefficient of $\tilde R_i\ (1\le i\le p)$ are contained in $\mathcal S$. Moreover, all of the above assertions stated for generic points $z$ remain valid for every $z\notin\mathcal S$.

Choose $\xi_u-p$ nonzero monic monomials $v_1,\ldots,v_{\xi_u-p}$ of degree $u$ in the variables $y_1,\ldots,y_q$ such that $\{\tilde R_1,\ldots,\tilde R_p,v_1,\ldots,v_{\xi_u-p}\}$ forms a $\mathcal K_{\mathcal Q}$-basis of $\mathcal K_{\mathcal Q}[y_1,\ldots,y_q]_u.$ Finally, denote by $\mathcal T=\{T_1,\ldots,T_{\xi_u}\}$ the set of all nonzero monic monomials of degree $u$ in the variables $y_1,\ldots,y_q$. Then $\{T_1,\ldots,T_{\xi_u}\}$ forms a $\mathcal K_{\mathcal Q}$-basis of $\mathcal K_{\mathcal Q}[y_1,\ldots,y_q]_u,$ and also a $\C$-basis of $\C[y_1,\ldots,y_q]_u.$

From \cite[Claim 4.3]{Q22c}, we have the following claim.
\begin{claim}\label{3.1}
There is an analytic proper subset $\mathcal S'$ of $\C^m$ such that for all $z\not\in\mathcal S'$:
\begin{itemize}
\item[(i)] the family of equivalent classes of $v_1,\ldots,v_{\xi_u-p}$ is a basis of $Y_{z,u}$ and the family $\{\tilde R_1(z),\ldots,\tilde R_p(z)\}$ is a $\C$-basis of $(I_{Y_z})_u$;
\item[(ii)] for a subset $\{v_1',\ldots,v'_{\xi_u-p}\}$ of $\mathcal T$, if $\{\tilde R_1,\ldots,\tilde R_p,v_1',\ldots,v'_{\xi_u-p}\}$ is a $\mathcal K_{\mathcal Q}$-basis of $\mathcal K_{\mathcal Q}[y_1,\ldots,y_q]_u$ then the set of equivalent classes of $v'_1,\ldots,v'_{\xi_u-p}$ modulo $(I_{Y_z})_u$ is a $\C$-basis of $Y_{z,u}$ for everey $z\not\in\mathcal S$;
\item[(iii)] otherwise if $\{\tilde R_1,\ldots,\tilde R_p,v_1',\ldots,v'_{\xi_u-p}\}$ is linearly dependent over $\mathcal K_{\mathcal Q}$ then the set of equivalent classes of $v_1',\ldots,v'_{\xi_u-p}$ modulo $(I_{Y_z})_u$ is not a $\C$-basis of $Y_{z,u}$.
\end{itemize}
\end{claim}
\noindent
Then, $\xi_u-p=H_{Y_z}(u)$ for all $z$ outside $\mathcal S\cup\mathcal S'.$ 

Consider the holomorphic map $F$ from $\C^m$ into $\P^{\xi_u-p-1}(\C)$ with the representation
$$ \tilde F=(v_1(\tilde\Phi\circ \tilde f),\ldots,v_{\xi_u-p}(\tilde\Phi\circ \tilde f)). $$
Since $f$ is algebraically nondegenerate over $\mathcal K_{\mathcal Q}$, $F$ is linearly nondegenerate over $\mathcal K_{\mathcal Q}$.

For each $z\not\in\mathcal S\cup\mathcal S'$, we set ${\bf c}_z = (c_{1,z},\ldots,c_{q,z})\in\mathbb Z^{q},$ where
\begin{align*}
c_{i,z}:=\log\frac{\|\tilde f(z)\|^d\|Q_i(z)\|}{|Q_i(\tilde f)(z)|}\ge 0, \text{ for } i=1,\ldots,q.
\end{align*}
By the definition of the Hilbert weight, there are ${\bf a}_{1,z},\ldots,{\bf a}_{\xi_u-p,z}\in\mathbb N^{q}$ with
$$ {\bf a}_{i,z}=(a_{i,1,z},\ldots,a_{i,q,z}), $$
where $a_{i,j,z}\in\{1,\ldots,\xi_u\},$ such that the residue classes modulo $(I_Y)_u$ of ${\bf y}^{{\bf a}_{1,z}},\ldots,{\bf y}^{{\bf a}_{\xi_u-p,z}}$ form a basic of $\C[y_1,\ldots,y_q]_u/(I_{Y_z})_u$ and
\begin{align*}
S_Y(u,{\bf c}_z)=\sum_{i=1}^{\xi_u-p}{\bf a}_{i,z}\cdot{\bf c}_z.
\end{align*}
Since ${\bf y}^{{\bf a}_{i,z}}\in\mathcal T$ and $\{\tilde R_1,\ldots,\tilde R_p,{\bf y}^{{\bf a}_{1,z}},\ldots,{\bf y}^{{\bf a}_{\xi_u-p,z}}\}$ is a basis of $\mathcal K_{\mathcal Q}[y_1,\ldots,y_q]$ (by Claim \ref{3.1}(iii)), the set of equivalent classes of $\{{\bf y}^{{\bf a}_{1,z}},\ldots,{\bf y}^{{\bf a}_{\xi_u-p,z}}\}$ is a basis of $\dfrac{\mathcal K_{\mathcal Q}[y_1,\ldots,y_q]_u}{I(Y)_u}$. Then $ {\bf y}^{{\bf a}_{i,z}}=L_{i,z}(v_1,\ldots,v_{H_Y(u)})\ \text{modulo }I(Y)_u, $ where $L_{i,z}\ (1\le i\le \xi_u-p)$ are $\mathcal K_{\mathcal Q}$-independent linear forms with coefficients in $\mathcal K_{\mathcal Q}$.
Therefore,
\begin{align*}
\log\prod_{i=1}^{\xi_u-p} |L_{i,z}(\tilde F(z))|&=\log\prod_{i=1}^{\xi_u-p}\prod_{j=1}^q|Q_j(\tilde f)(z)|^{a_{i,j,z}}\\
&= -S_Y(u,{\bf c}_z)+du(\xi_u-p)\log \|\tilde f(z)\| +\log C_1(z),
\end{align*}
where $C_1\in\mathcal C_{\mathcal Q}$. Note that the number of such linear forms $L_{i,z}$ is finite, and is at most $\xi_u$. Denote by $\mathcal L$ the collection of all linear forms $L_{i,z}$ appearing in the above inequalities. Then the above inequalities imply that
\begin{align*}
\log\prod_{i=1}^{\xi_u-p}\dfrac{\|\tilde F(z)\|\cdot \|L_{i,z}\|}{|L_{i,z}(\tilde F(z))|}= &S_Y(u,{\bf c}_z)-du(\xi_u-p)\log \|\tilde f(z)\| \\
&+(\xi_u-p)\log \|\tilde F(z)\|+\log C_2,
\end{align*}
where $C_2(z)\in\mathcal C_{\mathcal Q}$. Then, we have
\begin{align}\label{3.2}
\begin{split}
S_Y(u,{\bf c}_z)\le&\max_{\mathcal J\subset\mathcal L}\log\prod_{L\in \mathcal J}\dfrac{\|\tilde F(z)\|\cdot \|L\|}{|L(\tilde f(z))|}+du(\xi_u-p)\log \|\tilde f(z)\|\\
& -(\xi_u-p)\log \|\tilde F(z)\|+\log C_2(z),
\end{split}
\end{align}
where the maximum is taken over all subsets $\mathcal J\subset\mathcal L$ with $\sharp\mathcal J=\xi_u-p$ and $\{L|L\in\mathcal J\}$ is linearly independent over $\mathcal K_{\mathcal Q}$.
From Theorem \ref{2.1} we have
\begin{align}\label{3.3}
\dfrac{1}{u(\xi_u-p)}S_{Y_z}(u,{\bf c}_z)\ge\frac{1}{(n_f+1)\delta}e_{Y_z}({\bf c}_z)-\frac{(2n_f+1)\delta}{u}\max_{1\le i\le q}c_{i,z}
\end{align}
Combining (\ref{3.2}) and (\ref{3.3}), we get
\begin{align}\label{3.4}
\begin{split}
\frac{1}{(n_f+1)\delta}e_{Y_z}({\bf c}_z)
&\le\dfrac{1}{u(\xi_u-p)}\max_{\mathcal J\subset\mathcal L}\log\prod_{L\in \mathcal J}\dfrac{\|\tilde F(z)\|\cdot \|L\|}{|L(\tilde F(z))|}\\
&\ \ \ \ +\frac{(2n_f+1)\delta}{u}\max_{1\le i\le q}c_{i,z}+\log^+C_3(z)\\
&\le\dfrac{1}{u(\xi_u-p)}\max_{\mathcal J\subset\mathcal L}\prod_{L\in\mathcal J}\dfrac{\|\tilde F(z)\|\cdot \|L\|}{|L(\tilde F(z))|}\\
&\ \ \ \ +\frac{(2n_f+1)\delta}{u}\sum_{1\le i\le q}\log\frac{\|\tilde f(z)\|^d\|Q_i(z)\|}{|Q_i(\tilde f)(z)|}+\log^+C_3(z),
\end{split}
\end{align}
for generic points $z$, where $C_3(z)\in\mathcal C_{\mathcal Q}$.

Fix a point $z\not\in\mathcal S\cup \mathcal S'$. Choose $i\in\{1,\ldots,n_0\}$ such that
$$ e_{\sigma_i(1),z}\le e_{\sigma_i(2),z}\le\cdots\le e_{\sigma_i(q),z}.$$
Since $\bigcap_{j=1}^{\ell_i-1}Q_{\sigma_i(j)}(z)^*\cap V_z\ne\emptyset$ for generic points $z$, by Lemma \ref{2.2}, we have
\begin{align}\label{3.5}
\begin{split}
\Delta_{\mathcal Q,f}e_{Y_z}({\bf c}_z)&\ge (c_{\sigma_i(0),z}+\cdots +c_{\sigma_i(\ell_i),z})\cdot\delta\\
&=\delta\biggl(\sum_{j=1}^{\ell_i}\log\frac{\|\tilde f(z)\|^d\|Q_{\sigma_i(j)}(z)\|}{|Q_{\sigma_i(j)}(\tilde f)(z)|}\biggl ).
\end{split}
\end{align}
Then, from (\ref{3.2}), (\ref{3.4}) and (\ref{3.5}) we have
\begin{align}\label{3.6}
\begin{split}
\frac{1}{\Delta_{\mathcal Q,f}}\log \prod_{i=1}^q\dfrac{\|\tilde f (z)\|^d\|Q_i(z)\|}{|Q_i(\tilde f)(z)|}
&\le\dfrac{n_f+1}{u(\xi_u-p)}\max_{\mathcal J\subset\mathcal L}\log\prod_{L\in\mathcal J}\dfrac{\|\tilde F(z)\|\cdot \|L\|}{|L(\tilde F(z))|}\\
&+\frac{(2n_f+1)(n_f+1)\delta}{u}\sum_{1\le i\le q}\log\frac{\|\tilde f(z)\|^d\|Q_{i}(z)\|}{|Q_{i}(\tilde f)(z)|}\\
&+\frac{1}{\Delta_{\mathcal Q,f}}\log C(z)+(n_f+1)\log^+C_3(z)
\end{split}
\end{align}
for generic points $z$.

Denote by $\Psi$ the set of all coefficients of the linear forms $L_{i,z}$, and write $\Psi=\{a_1,\ldots,a_{q_0}\}.$
Then $\sharp\mathcal L\le \xi_u$ and $\sharp\Psi=q_0\le \xi_u(\xi_u-p).$
For each positive integer $m$, let $\mathcal L(\Psi(m))$ denote the $\C$-vector space generated by the set
$$\left\{a_1^{i_1}\cdots a_{q_0}^{i_{q_0}}\,\middle|\,i_j\ge 0,\ \sum_{j=1}^{q_0}i_j\le m\right\}.$$
By Remark 3.4 in \cite{QT10}, there exists a smallest positive integer $p'$ such that
$$ \frac{\dim\mathcal L(\Psi (p'+1))}{\dim\mathcal L(\Psi (p'))}\le 1+\frac{\epsilon}{2\Delta_{\mathcal Q,f}(n+1)}.$$
Put $s=\dim\mathcal L(\Psi (p')), t=\dim\mathcal L(\Psi (p'+1))$. Again, by  \cite[Remark 3.4]{QT10}, we have
\begin{align*}
t&\le \left\lfloor\left(1+\frac{\epsilon}{2(n_f+1)\Delta_{\mathcal Q,f}}\right)^{\left\lfloor\frac{\sharp\Psi}{\log^2(1+\frac{\epsilon}{2(n_f+1)\Delta_{\mathcal Q,f}})}\right\rfloor+1}\right\rfloor\\
&\le \left\lfloor\left(1+\frac{\epsilon}{2(n_f+1)\Delta_{\mathcal Q,f}}\right)^{\left\lfloor\frac{d^{n_f}\delta_f(u+1)^{n_f+q}}{\log^2(1+\frac{\epsilon}{2(n_f+1)\Delta_{\mathcal Q,f}})}\right\rfloor+1}\right\rfloor.
\end{align*}
Here, the last inequality comes from the fact that $\xi_u\le (u+1)^q$ and 
$$\xi_u-p\le\delta\binom{n_f+u}{n_f}\le d^{n_f}\delta_f\binom{n_f+u}{n_f}.$$

Choose a $\C$-basis $\{b_1,\ldots,b_s\}$ of $\mathcal L(\Psi(p'))$, and extend it to a $\C$-basis $\{b_1,\ldots,b_t\}$ of $\mathcal L(\Psi(p'+1))$.
Consider the meromorphic mapping $\tilde F:\C^m\longrightarrow \mathbb P^{t(\xi_u-p)-1}(\C)$ with a reduced representation
$$\tilde F=(b_1v_1(\tilde\Phi\circ \tilde f),\ldots,b_1v_{\xi_u-p}(\tilde\Phi\circ \tilde f),\ldots, b_tv_1(\tilde\Phi\circ \tilde f),\ldots,b_tv_{\xi_u-p}(\tilde\Phi\circ \tilde f)).$$
Note that $\|\ T_{\tilde F}(r)=duT_f(r)+o(T_f(r))$. By the general form of the second main theorem for fixed hyperplanes, we have
\begin{align}\label{3.7}
\begin{split}
\biggl \|\ s\int_0^{2\pi}\max_{\mathcal J\subset\mathcal L}\log\prod_{L\in\mathcal J}\frac{\|\tilde F\|}{|L(\tilde F)|}\frac{d\theta}{2\pi}-N_{W(\tilde F)}(r)&\le t(\xi_u-p)T_{\tilde F}(r)+o(T_{\tilde F}(r))\\
&=t(\xi_u-p)udT_f(r)+o(T_f(r))
\end{split}
\end{align}
where $\max_{\mathcal J\subset\mathcal L}$ is taken over all subsets $\mathcal J$ of the system $\mathcal L$ of linear forms such that $\mathcal J$ is linearly independent over $\mathcal {K}_{\mathcal {Q}}$, $\epsilon'$ is an arbitrary positive number. Integrating (\ref{3.6}) and using (\ref{3.7}), we obtain
\begin{align*}
\biggl\|\ &\left(\frac{1}{\Delta_{\mathcal Q,f}}-\frac{(2n_f+1)(n_f+1)\delta}{u}\right)\sum_{i=1}^qm_f(r,Q_i)\\
&\le\frac{d(n_f+1)t}{s}T_f(r)-\frac{(n_f+1)}{u(\xi_u-p)s}N_{W(\tilde F)}(r)+o(T_f(r)).
\end{align*}

We now estimate the quantity $N_{W(\tilde F)}(r)$. Let $z\in\C^m\setminus\{\mathcal S\cup\mathcal S'\}$ which is neither zero nor pole of any coefficients of $Q_i\ (1\le i\le q)$ and $b_i\ (1\le i\le t)$. We set 
$$c_{i}=\max\{0,\nu^0_{Q_i(\tilde f)}(z)-\xi_u+p+1\}\ (1\le i\le q) \text{ and } {\bf c}=(c_{1},\ldots,c_{q})\in\mathbb Z^q_{\ge 0}.$$
Then there are 
$${\bf a}_i=(a_{i,1},\ldots,a_{i,q}),a_{i,s}\in\{1,...,u\}$$
such that ${\bf y}^{{\bf a}_1},...,{\bf y}^{{\bf a}_{H_Y(u)}}$ is a basic of $\C[y_1,\ldots,y_q]_u/(I_{Y_z})_u$ and
$$ S_{Y_z}(u,{\bf c})=\sum_{i=1}^{H_Y(u)}{\bf a}_i\cdot{\bf c}.$$
Similarly as above, we write ${\bf y}^{{\bf a}_i}=L_i(v_1,...,v_{\xi_u-p})$, where $L_1,...,L_{\xi_u-p}$ are linearly independent linear forms. We see that 
$$ \nu^0_{W(\tilde F)}(z)\ge t\sum_{i=1}^{H_Y(u)}\max\{0,\nu^0_{L_i(\tilde F)}(z)-n_u\},$$
where $n_u=(\xi_u-p)t-1$. It is easy to see that
$$ \nu^0_{L_i(\tilde F)}(z)=\sum_{j=1}^qa_{i,j}\nu^0_{Q_j(\tilde f)}(z),$$
and hence
$$ \max\{0,\nu^0_{L_i(\tilde F)}(z)-n_u\}\ge\sum_{j=1}^qa_{i,j}c_{j}={{\bf a}_i}\cdot{\bf c}. $$
Thus, 
\begin{align}\label{3.8}
\nu^0_{W(\tilde F)}(z)\ge t\sum_{i=1}^{H_Y(u)}{{\bf a}_i}\cdot{\bf c}=tS_Y(u,{\bf c}).
\end{align}
Choose an index $\sigma_{i_0}$ such that 
$$\nu^0_{Q_{\sigma_{i_0}(1)}(\tilde f)}(z)\ge \nu^0_{Q_{\sigma_{i_0}(2)}(\tilde f)}(z)\ge\cdots\ge \nu^0_{Q_{\sigma_{i_0}(q)}(\tilde f)}(z).$$
By Lemma \ref{2.2}, we have
\begin{align*}
\Delta_{\mathcal Q,f}e_{Y_z}({\bf c})&\ge (c_{\sigma_{i_0}(1),z}+\cdots +c_{\sigma_{i_0}(\ell_{i_0}),z})\cdot\delta\\
&=\delta\cdot\sum_{j=1}^{\ell_{i_0}}\max\{0,\nu^0_{Q_{\sigma_{i_0}(j)}(\tilde f)}(z)-n_u\}\\
&=\delta\cdot\sum_{j=1}^{q}\max\{0,\nu^0_{Q_j(\tilde f)}(z)-n_u\}+O(\nu_{R^{i_0}}(z)),
\end{align*}
where $R^{i}$ is the resultant of the family $\{Q_{\sigma_i(j)}\}_{j=1}^{\ell_i}$ for $i=1,\ldots,q$. 

On the other hand, by Theorem \ref{2.1} we have that 
\begin{align*}
 S_{Y_z}(u,{\bf c}) &\ge\frac{u(\xi_u-p)}{(n_f+1)\delta}e_Y({\bf c})-(2n_f+1)\delta(\xi_u-p)\max_{1\le i\le q}c_{i}+O(\nu_{R^{i_0}}(z))\\
&\ge\left(\frac{u(\xi_u-p)}{\Delta_{\mathcal Q,f}(n_f+1)}-(2n_f+1)\delta(\xi_u-p)\right)\\
&\ \ \times \sum_{j=1}^{q}\max\{0,\nu^0_{Q_j(\tilde f)}(z)-n_u\}+O(\nu_{R^{i_0}}(z)).
\end{align*}
Combining this inequality and (\ref{3.8}), we have
\begin{align*}
\dfrac{(n_f+1)}{u(\xi_u-p)s}\nu^0_{W(\tilde F)}(z)&\ge\dfrac{t}{us}\left(\frac{u}{\Delta_{\mathcal Q,f}}-(2n_f+1)(n_f+1)\delta_f\right )\\
&\times\sum_{j=1}^{q}\max\{0,\nu^0_{Q_j(\tilde f)}(z)-n_u\}+O(\sum_{i=1}^{n_0}\nu_{R^{i}}(z)).
\end{align*}
Integrating both sides of this inequality, we obtain 
\begin{align*}
\biggl \|\ \dfrac{(n_f+1)}{u(\xi_u-p)s}N_{W(\tilde F)}(r)&\ge \dfrac{t}{s}\left(\frac{1}{\Delta_{\mathcal Q,f}}-\frac{(2n_f+1)(n_f+1)\delta}{u}\right )\\
&\ \ \times\sum_{j=1}^{q}\left(N_{(Q_j,f)}(r)-N^{[n_u]}_{(Q_j,f)}(r)\right)+o(T_f(r)).
\end{align*}
Seting $m_0=\frac{1}{\Delta_{\mathcal Q,f}}-\frac{(2n_f+1)(n_f+1)\delta}{u}$ and combining inequalities (\ref{3.7}) with the above inequality, we get
\begin{align*}
\biggl\|\ \sum_{i=1}^{q}m_f(r,Q_i)\le \frac{d(n_f+1)t}{sm_0}T_f(r)-\frac{t}{s}\sum_{j=1}^{q}\left(N_{(Q_i,f)}(r)-N^{[n_u]}_{(Q_j,f)}(r)\right)+o(T_f(r)).
\end{align*}
This inequality implies that
\begin{align}\label{3.9}
\biggl\|\ \left(q-\frac{(n_f+1)t}{sm_0}\right)T_f(r)\le\sum_{j=1}^{q}\frac{1}{d}N^{[n_u]}_{(Q_j,f)}(r)+o(T_f(r)).
\end{align}

\noindent
We choose $u=\lceil 2\Delta_{\mathcal Q,f}(2n_f+1)(n_f+1)d^{n_f}\delta_f(\Delta_{\mathcal Q,f}(n_f+1)+\epsilon)\epsilon^{-1}\rceil$. 
Then 
$$u\ge \biggl\lceil \dfrac{\Delta_{\mathcal Q,f}(2n_f+1)(n_f+1)\delta(\Delta_{\mathcal Q,f}(n_f+1)+\epsilon)}{\Delta_{\mathcal Q,f}(n_f+1)+\epsilon-\frac{\Delta_{\mathcal Q,f}(n_f+1)t}{s}}\biggl\rceil,$$ 
and we have:
\begin{align*}
\bullet\ &q-\frac{(n_f+1)t}{sm_0}\ge q-\frac{\Delta_{\mathcal Q,f}(n_f+1)t}{(1-\Delta_{\mathcal Q,f}(2n_f+1)(n_f+1)\delta/u)s}\\
&\ge q-\Delta_{\mathcal Q,f}(n_f+1)-\epsilon;\\
\bullet\ &n_u+1=(\xi_u-p)t\\
&\le  d^{n_f}\delta_f(u+1)^{n_f}\left\lfloor\bigl(1+\frac{\epsilon}{2(n_f+1)\Delta_{\mathcal Q,f}}\bigl)^{\bigl[\frac{d^{n_f}\delta_f(u+1)^{n_f+q}}{\log^2(1+\frac{\epsilon}{2(n_f+1)\Delta_{\mathcal Q,f}})}\bigl]+1}\right\rfloor=L;\\
\bullet\ &\frac{(t(\xi_u-p)-1)t(n_f+1)}{2sdm_0u}< \frac{t(n_f+1)}{sm_0}\cdot\frac{(L-1)}{2du}.
\end{align*}
Then, from (\ref{3.9}) we have
\begin{align*}
\bigl\|\ (q-\Delta_{\mathcal Q,f}(n_f+1)-\epsilon)T_f(r)\le\sum_{j=1}^{q}\frac{1}{d}N^{[L-1]}_{(Q_j,f)}(r)+o(T_f (r)).
\end{align*}
Then, we get the desired inequality.

If all $Q_i$ are fixed hypersurfaces then $t=s=1$. By choosing 
$$u'=\lceil \Delta_{\mathcal Q,f}(2n_f+1)(n_f+1)d^{n_f}\delta_f(\Delta_{\mathcal Q,f}(n_f+1)+\epsilon)\epsilon^{-1}\rceil$$ 
and replacing $u$ in the above by $u'$, we have
\begin{align*}
n_{u'}+1&=\xi_{u'}-p<\delta\binom{n_f+u'}{n_f}\le d^{n_f}\delta_f\binom{n_f+u'}{n_f}\\
&\le \left\lfloor d^{n_f}\delta_fe^{n_f}\left(1+\frac{u'}{n_f}\right)^{n_f}\right\rfloor\\
&\le \left\lfloor d^{{n_f}^2+n_f}(\delta_f)^{n_f+1}e^{n_f}(2n_f+5)^{n_f}(\Delta^2_{\mathcal Q,f}(n_f+1)\epsilon^{-1}+\Delta_{\mathcal Q,f})^{n_f}\right\rfloor,
\end{align*}
and hence we may get the truncation level
$$ L=\left\lfloor^{n_f^2+n_f}(\delta_f)^{n_f+1}e^{n_f}(2n_f+5)^{n_f}(\Delta^2_{\mathcal Q,f}(n_f+1)\epsilon^{-1}+\Delta_{\mathcal Q,f})^{n_f}\right\rfloor.$$
Hence, the proof of the theorem is completed.
\end{proof}

\section{Proof of Main Theorems}
Motivated by the idea of using Nochka diagram for families of fixed hypersurfaces in subgeneral position of G. Heier and A. Levin in \cite{HL}, we now prove the following lemma.

\begin{lemma}\label{4.1}
Let $f$ be a nonconstant meromorphic mappings from $\C^m$ into a projective subvariety $V\subset\P^n(\C)$ of dimension $k$. Let $\mathcal{Q} = \{Q_1, \ldots, Q_q\}$ be a family of moving hypersurfaces in $\P^n(\C)$ which has distributive constant $\Delta_{\mathcal Q,V}$ with respect to $V\ (N> k)$ and satisfying the B\'ezout property with respect to $f$. Let $n_f$ be the algebraic dimension of $f$ over the field $\mathcal K_{\mathcal Q}$. Then there exists a subset $\Gamma \subset \{1,\ldots,q\}$ such that $\sharp\Gamma\ge q -\Delta_{\mathcal Q,V}\left(k-\left\lceil\frac{n_f}{2}\right\rceil\right)$ and the family $\mathcal P= \{Q_i : i \in \Gamma\}$ satisfies $\Delta_{\mathcal P,f}\le \frac{2k-n_f+1}{n_f+1}\Delta_{\mathcal Q,V}$.
\end{lemma}

\begin{proof} 
Let $\{R_1,\ldots,R_\lambda\}$ a minimal subset of $\mathcal K_{\mathcal Q}[x_0,\ldots,x_n]$, which generates $I(f(\C^m))$. 
For $z$ is not a pole of any $R_i\ (1\le i\le \lambda)$, set $V_z=\bigcap_{i=1}^\lambda R_i(z)^*$. Then $\dim V_z=n_f$ for generic points $z$.

For any subset $\Gamma \subset \{1,\ldots,q\}$, denote
$$Q_\Gamma(z) = \bigcap_{i \in \Gamma}Q_i(z)^* \cap V_z, \quad c(\Gamma) = \operatorname{codim}_{V_z} Q_\Gamma(z)$$
for generic points $z$. Note that
$$\Delta_{\mathcal Q,f}= \max_{\Gamma \ne \emptyset} \frac{\sharp\Gamma}{c(\Gamma)}.$$
If $\Delta_{\mathcal Q,f}\le \frac{2k-n_f+1}{n_f+1}\Delta_{\mathcal Q,V}$, then the conclusion holds with $\Gamma = \{1,\ldots,q\}$.

Assume $\Delta_{\mathcal Q,f}>\frac{2k-n_f+1}{n_f+1}\Delta_{\mathcal Q,V}$. Let
$$\mathcal{S} = \left\{ \Gamma \subset \{1,\ldots,q\} : \frac{\sharp\Gamma}{c(\Gamma)} >\frac{2k-n_f+1}{n_f+1}\right\}.$$
Then $\mathcal{S}$ is nonempty. Choose $\Gamma_1 \in \mathcal{S}$ such that $\sharp\Gamma_1$ is maximal. Since 
\begin{align*}
\sharp\Gamma_1&\le \Delta_{\mathcal Q,V}\left(\dim V-\dim V\cap\bigl(\bigcap_{j\in\Gamma_1}Q_j(z)^*\bigl)\right)\\
&\le \Delta_{\mathcal Q,V}\left(\dim V-\dim V_z\cap\bigl(\bigcap_{j\in\Gamma_1}Q_j(z)^*\bigl)\right)\\
&=\Delta_{\mathcal Q,V}\left(k-n_f+c(\Gamma_1)\right),
\end{align*} 
we have
$$\frac{2k-n_f+1}{n_f+1}\Delta_{\mathcal Q,V}<\frac{\sharp\Gamma_1}{c(\Gamma_1)}\le \frac{\left(k-n_f+c(\Gamma_1)\right)}{c(\Gamma_1)}\Delta_{\mathcal Q,V}.$$
This implies that $c(\Gamma_1)<\frac{n_f+1}{2}$ and then $c(\Gamma_1)\le \left\lfloor\frac{n_f}{2}\right\rfloor$.

\medskip
\noindent \textbf{Claim.} For any $\Gamma_2 \in \mathcal{S}\setminus\{\Gamma_1\}$, we have $\Gamma_1 \cap \Gamma_2 \ne \emptyset$.
\medskip

\noindent \textit{Proof of Claim.}
Assume $\Gamma_1 \cap \Gamma_2 = \emptyset$. Similar as above, one has $c(\Gamma_2)<\frac{n_f+1}{2}$.  Using the B\'ezout property, we have
$$c(\Gamma_1 \cup \Gamma_2) \le c(\Gamma_1) + c(\Gamma_2)<n_f+1.$$
This yields that $\bigcap_{j\in \Gamma_1 \cup \Gamma_2}(Q_j(z)^*\cap V_z)\ne\varnothing$ for generic points $z$.

Also, from
$$\frac{\sharp\Gamma_1}{c(\Gamma_1)} > \frac{2k-n_f+1}{n_f+1}\Delta_{\mathcal Q,V} \quad\text{ and }\quad \frac{\sharp\Gamma_2}{c(\Gamma_2)} > \frac{2k-n_f+1}{n_f+1}\Delta_{\mathcal Q,V},$$
by using the B\'ezout property we obtain
$$\frac{\sharp\Gamma_1 + \sharp\Gamma_2}{c(\Gamma_1 \cup \Gamma_2)}\ge \frac{\sharp\Gamma_1+\sharp\Gamma_2}{c(\Gamma_1) + c(\Gamma_2)} > \frac{2k-n_f+1}{n_f+1}\Delta_{\mathcal Q,V}.$$
Thus $\Gamma_1 \cup \Gamma_2 \in \mathcal{S}$, contradicting the maximality of $\sharp\Gamma_1$. 
The claim is proved.

Now set
$$\Gamma = \{1,\ldots,q\} \setminus \Gamma_1, \quad \mathcal{P} = \{Q_i : i \in \Gamma\}.$$
Then
$$\sharp\Gamma\ge q -\Delta_{\mathcal Q,V}\left(k-n_f+c(\Gamma_1)\right)\ge q -\Delta_{\mathcal Q,V}\left(k-\left\lceil\frac{n_f}{2}\right\rceil\right).$$
Finally, suppose there exists $\Gamma' \subset \Gamma$ such that
$$
\frac{|\Gamma'|}{c(\Gamma')} > \frac{2k-n_f+1}{n_f+1}\Delta_{\mathcal Q,V}.
$$
Then $\Gamma' \in \mathcal{S}$ but $\Gamma' \cap \Gamma_1 = \emptyset$, contradicting the claim. 
Hence
$$\Delta_{\mathcal P,V}\le \frac{2k-n_f+1}{n_f+1}\Delta_{\mathcal Q,V}.$$
This completes the proof.
\end{proof}

If the family $\mathcal{Q} = \{Q_1, \ldots, Q_q\}$ is in $N$-subgeneral position then we will have more optimal result as follows. 
\begin{lemma}\label{4.2}
Let $f$ be a nonconstant meromorphic mappings from $\C^m$ into a projective subvariety $V\subset\P^n(\C)$ of dimension $k$. Let $\mathcal{Q} = \{Q_1, \ldots, Q_q\}$ be a family of moving hypersurfaces in $\P^n(\C)$ in $N$-subgeneral position with respect to $V\ (N> k)$ and satisfying the B\'ezout property with respect to $f$. Let $n_f$ be the algebraic dimension of $f$ over the field $\mathcal K_{\mathcal Q}$. Then there exists a subset $\Gamma \subset \{1,\ldots,q\}$ such that $\sharp\Gamma\ge q -N+\left\lceil\frac{n_f}{2}\right\rceil$ and the family $\mathcal P= \{Q_i : i \in \Gamma\}$ satisfies $\Delta_{\mathcal P,f}\le \frac{2N-n_f+1}{n_f+1}$.
\end{lemma}

\begin{proof} 
We use the same notation and the arguments as in the proof of Lemma \ref{4.1}. If $\Delta_{\mathcal Q,f}\le \frac{2N-n_f+1}{n_f+1}$ then the conclusion holds with $\Gamma = \{1,\ldots,q\}$.

Assume $\Delta_{\mathcal Q,f}>\frac{2N-n_f+1}{n_f+1}$. Let
$$\mathcal{S} = \left\{ \Gamma \subset \{1,\ldots,q\} : \frac{\sharp\Gamma}{c(\Gamma)} >\frac{2N-n_f+1}{n_f+1}\right\}.$$
Then $\mathcal{S}$ is nonempty. Choose $\Gamma_1 \in \mathcal{S}$ such that $\sharp\Gamma_1$ is maximal. Since 
$\sharp\Gamma_1\le N-n_f+c(\Gamma_1),$
we have
$$\frac{2N-n_f+1}{n_f+1}<\frac{\sharp\Gamma_1}{c(\Gamma_1)}\le \frac{N-n_f+c(\Gamma_1)}{c(\Gamma_1)}.$$
This implies that $c(\Gamma_1)<\frac{n_f+1}{2}$ and then $c(\Gamma_1)\le \left\lfloor\frac{n_f}{2}\right\rfloor$.

Similarly as the proof of Lemma \ref{4.1}, we have that $\Gamma_1 \cap \Gamma_2 \ne \emptyset$ for any $\Gamma_2 \in \mathcal{S}\setminus\{\Gamma_1\}$.

By setting
$$\Gamma = \{1,\ldots,q\} \setminus \Gamma_1, \quad \mathcal{P} = \{Q_i : i \in \Gamma\},$$
we have
$$\sharp\Gamma\ge q -N+n_f+c(\Gamma_1)\ge q -N+\left\lceil\frac{n_f}{2}\right\rceil,$$
and $\Gamma'\not\in\mathcal S$ for every $\Gamma' \subset \Gamma$ (since $\Gamma' \cap \Gamma_1 = \emptyset$). Hence
$$\Delta_{\mathcal P,V}\le \frac{2N-n_f+1}{n_f+1}.$$
This completes the proof.
\end{proof}

\begin{proof}[Proof of Theorem \ref{1.1}]
By Lemma \ref{4.1},  there is a subset $\Gamma$ of $\{1,\ldots,q\}$ with $\sharp\Gamma\ge q -\Delta_{\mathcal Q,V}\left(k-\left\lceil\frac{n_f}{2}\right\rceil\right)$ such that the family of hypersurfaces $\mathcal P=\{Q_i;i\in\Gamma\}$ satisfies $\Delta_{\mathcal P,V}\le \frac{2k-n_f+1}{n_f+1}\Delta_{\mathcal Q,V}$. Applying Theorem \ref{main} for $f$ and the family $\mathcal P$, we obtain
$$\biggl \|\ (q -\Delta_{\mathcal Q,V}\left(2k-n_f+1+k-\left\lceil\frac{n_f}{2}\right\rceil\right)-\epsilon)T_f(r)\le\sum_{i\in\Gamma_1}\frac{1}{\deg Q_i}N^{[L_1'-1]}(r,\nu_{(Q_i,f)})+o(T_f(r)),$$
where 
$$L_1'=d^{n_f}\delta_f(u'+1)^{n_f}\left\lfloor\bigl(1+\frac{\epsilon}{2(n_f+1)\Delta_{\mathcal P,f}}\bigl)^{\left\lfloor\frac{d^{n_f}\delta_f(u'+1)^{n_f+q}}{\log^2(1+\frac{\epsilon}{2(n_f+1)\Delta_{\mathcal P,f}})}\right\rfloor+1}\right\rfloor$$
with
$$u'=\lceil 2\Delta_{\mathcal P,f}(2n_f+1)(n_f+1)d^{n_f}\delta_f(\Delta_{\mathcal P,f}(n_f+1)+\epsilon)\epsilon^{-1}\rceil.$$
We have the following estimate:
\begin{itemize}
\item $1\le n_f\le k$;
\item $\Delta_{\mathcal P,f}(n_f+1)\le (2k-n_f+1)\Delta_{\mathcal Q,V}\le 2k\Delta_{\mathcal Q,V}$;
\item $\log (1+x)>\frac{2x}{x+2}\ \forall x>0$, and hence $(1+x)^{\frac{1}{\log^2(1+x)}}=e^{\frac{1}{\log (1+x)}}\le e^{\frac{1}{x}+\frac{1}{2}}\ \forall x>0$.
\end{itemize}
Therefore, we have
\begin{itemize}
\item $u'\le \lceil 2k\Delta_{\mathcal Q,V}(2k+1)d^{k}\delta_f(2k\Delta_{\mathcal Q,V}+\epsilon)\epsilon^{-1}\rceil=u;$
\item $L'\le d^{k}\delta_f(u+1)^{k}\left\lfloor\bigl(1+\frac{\epsilon}{4}\bigl)e^{\left(d^k\delta_f(u+1)^{k+q}\right)\left(\frac{1}{2}+\frac{4}{\epsilon}\right)}\right\rfloor=L.$
\end{itemize}
Then, the above inequalities imply that
$$\biggl \|\ \left (q -\Delta_{\mathcal Q,V}(3k-1)-\epsilon\right)T_f(r)\le\sum_{i=1}^q\frac{1}{\deg Q_i}N^{[L-1]}(r,\nu_{(Q_i,f)})+o(T_f(r)).$$

Moreover, if all $Q_i\ (1\le i\le q)$ are fixed hypersurfaces then we may take 
\begin{align*}L'&=\left\lfloor d^{n_f^2+n_f}(\delta_f)^{n_f+1}e^{n_f}(2n_f+5)^{n_f}(\Delta^2_{\mathcal P,f}(n_f+1)\epsilon^{-1}+\Delta_{\mathcal P,f})^{n_f}\right\rfloor\\
&\le \left\lfloor d^{k^2+k}(\delta_f)^{k+1}e^{k}\frac{(2k+5)^{k}(2k)^k}{(k+1)^k}\Delta_{\mathcal Q,V}^k(2k\Delta_{\mathcal Q,V}\epsilon^{-1}+1)^k\right\rfloor\\
&\le \left\lfloor d^{k^2+k}(\delta_f)^{k+1}e^{k+\frac{3}{2}}(4k)^k\Delta_{\mathcal Q,V}^k(2k\Delta_{\mathcal Q,V}\epsilon^{-1}+1)^k\right\rfloor.
\end{align*}
Here, the last inequality comes from that 
$$\frac{(2k+5)^{k}(2k)^k}{(k+1)^k}=\left(4k+\frac{6k}{k+1}\right)^k<(4k+6)^k\le (4k)^ke^{3/2}.$$ 
Then we may take
$$ L=\left\lfloor d^{k^2+k}(\delta_f)^{k+1}e^{k+\frac{3}{2}}(4k)^k\Delta_{\mathcal Q,V}^k(2k\Delta_{\mathcal Q,V}\epsilon^{-1}+1)^k\right\rfloor.$$

Hence, the theorem is proved.
\end{proof}

\begin{proof}[Proof of Theorem \ref{1.2}]
We follow the same argument as in the proof of Theorem \ref{1.1}. The only difference is that, instead of applying Lemma \ref{4.1}, we apply Lemma \ref{4.2}, noting that
$$(n_f+1)\Delta_{\mathcal P,f}\le 2N.$$
Since the proofs of the two theorems are essentially identical, we omit the detailed proof of Theorem \ref{1.2}.
\end{proof}
\vskip0.2cm
\noindent
\section*{Data availability}
Data sharing not applicable to this article as no datasets were generated or analyzed during the current study.

\section*{Disclosure statement} 
The authors states that there is no conflict of interest.

\vskip0.2cm
{\footnotesize 
\noindent
$^1$ Department of Mathematics, School of Mathematics and Computational Science, Hanoi National University of Education,\\
 136-Xuan Thuy, Cau Giay, Hanoi, Vietnam.\\
$^2$ Institute of Natural Sciences, Hanoi National University of Education,\\
 136-Xuan Thuy, Cau Giay, Hanoi, Vietnam.\\
\textit{E-mail}: quangsd@hnue.edu.vn, linhchihnue2401@gmail.com}

\end{document}